\documentclass[10pt, oneside]{amsart}       
\usepackage{geometry}                        
\usepackage{graphicx}                
                            \usepackage{amsthm}    
\newtheorem{theorem}{Theorem}[section]
\newtheorem{definition}{Definition}[section]
\newtheorem{lemma}[theorem]{Lemma}
\theoremstyle{remark}
\newtheorem*{remark}{Remark}
\usepackage{tcolorbox}
\newtheorem*{example}{Example}

\title{Topology of Singularities on Algebraic Loop Spaces}
\author{EMILE BOUAZIZ}

\begin{document}
\maketitle

\begin{abstract} We study the topology of some simple infinite dimensional singularities arising from spaces of \emph{algebraic formal loops}. We prove that in some simple cases the natural analogue of nearby cycles cohomology for a function on the loop space vanishes, and show further that a suitably renormalized version of the above cohomology produces a simple (non-zero) result. \end{abstract}

\section{introduction} The goal of this note is to present some rather simple computations of the cohomology of some infinite dimensional spaces.  If $X$ is an affine variety (over $\mathbb{C}$) with a regular function $F:X\rightarrow\mathbb{A}^{1}$ which is topologically a fibration over $\mathbb{G}_{m}$, then a measure of the singularity of $F$ over $\{0\}$ is given by the functor of vanishing cycles and its close relative, nearby cycles. At a crude approximation (which will be sufficient for our purposes) these functors measure the difference in central and nearby fibre cohomology.

 If one restricts to the simple case where the central fibre is contractible, for example if $F$ is homogeneous, then to compute vanishing cycles one must simply compute the cohomology of the nearby fibre. According to a result of Milnor, and under the additional assumption that the singularity at the origin be isolated, vanishing cohomology of such a function is concentrated in degree $\dim X-1$, and of dimension the \emph{Milnor number}, $\mu(X)$. 

Attached to $X$ is a certain infinite dimensional space (in general not representable by a scheme, but still fairly manageable) of loops into $X$, denoted $\mathcal{L}X$. Further, the data of $F$ endows this space of loops with a natural global function, denoted $\Lambda(F)$ which can roughly be defined as the constant (in $t$) term obtained upon replacing local coordinate functions $z^{i}$, in $F(z^{1},...,z^{d})$, with Laurent series $z^{i}(t)=\sum_{j}z^{i}_{j}t^{j}$, so that the variables $z^{i}_{j}$ are functions on the loop space of $X$. Everything in this context is highly infinite dimensional and so one certainly should not expect a simple analogue of the entire vanishing cycles package, but it is of course still possible to compare central and nearby fibre cohomology, which will be the subject of this note.

\begin{example} A useful example to bear in mind is the pair $(\mathbb{A}^{1},z^{2})$. We have $$\Lambda(z^{2})=z_{0}^{2}+2\sum_{j>0}z_{j}z_{-j}\in\mathcal{O}(\mathcal{L}\mathbb{A}^{1}).$$ We think of this as an (ind-) infinite dimensional affine quadric. Truncating the Laurent tails by requiring no poles of order greater than $n$, we obtain the function $z_{0}^{2}+2\sum_{j=1}^{n}z_{j}z_{-j}$. The nearby fibre cuts out the product of a quadric in $\mathbb{A}^{2n+1}$  with an infinite dimensional affine space, and is thus homeomorphic to $T^{*}S^{2n}\times\mathbb{R}^{\infty}$. The nearby fibres are therefore homotopic to $S^{2n}$, so that in the limit we obtain a weakly contractible space. If we instead computed a colimit of maps along Gysin maps on cohomology we would obtain a one dimensional space in degree $0$, which we note agrees with vanishing cohomology of the pair $(\mathbb{A}^{1},z^{2})$. \end{example} 

We compute the relevant spaces explicitly in the case of homogeneous $F$ with an isolated singularity, that is to say the cone singularity for the affine cone on a smooth projective hypersurface. Further, we observe that, in the spirit of \cite{AK}, there is a natural method by which to renormalize the cohomology spaces, so as to avoid classes running off to infinity and thus vanishing in the limit. It is in fact this renormalized/ stabilized cohomology with which we deal primarily in this note. We remark that the methods presented here are quite elementary, we only make use of basic properties of Gysin maps and residues, the reader familiar with loop spaces and cohomology of ind-schemes more broadly can safely skip straight to Section 3, and indeed is encouraged to do so. 

We state here our main theorem, undefined terms will be defined in the following sections. \begin{tcolorbox}\begin{theorem} Let $F:\mathbb{A}^{d}\rightarrow\mathbb{A}^{1}$ be a homogeneous function with an isolated singularity at $0$. Then the renormalized nearby cohomology of the pair $(\mathcal{L}\mathbb{A}^{d},\Lambda(F))$ is concentrated in degree $d-1$ and of dimension $\mu(F)$. \end{theorem}\end{tcolorbox} 

We remark that the above computation agrees with the corresponding (reduced) nearby cohomology of the pair $(\mathbb{A}^{d},F)$, which maps to $(\mathcal{L}\mathbb{A}^{d},\Lambda(F))$ via the inclusion of constant loops into loops. As such, the above could be phrased as a cohomology vanishing theorem for the open complement.

Further, over the course of the proof of the above we will also prove the following, \begin{theorem} The reduced (co)homology of the nearby fibre to $\Lambda(F)$ vanishes. \end{theorem}

\subsection{Acknowledgements} Some of the results presented here were obtained when the author was a PhD student at the University of Cambridge under the supervision of Ian Grojnowski, whom the author wishes to warmly thank.

\section{basics}
We will recall here some basic notions in the theory of loop spaces. We will work throughout over the field \(\mathbb{C}\) of complex numbers. By a \emph{space} we will mean a pre-sheaf on the category of affine \(\mathbb{C}\)-schemes, so that to give a space is simply to give its value on all test algebras $A$. We will deal in particular with \emph{ind-schemes}, which are filtered colimits (taken inside the pre-sheaf category) of schemes along closed embeddings. Given a set \(S\), throughout we will write \(\mathbb{A}\{S\}\) for the affine space with ring of functions, $\mathbb{C}[S]$, the free commutative $\mathbb{C}$-algebra generated by $S$.

\subsection{Loop Spaces} We now define some of our basic objects of study, mostly to fix notation. \(X\) will denote an arbitrary space below.

\begin{definition} \begin{enumerate} \item The \emph{arc space} of \(X\), \(\mathcal{L}^{+}X\), is defined as the space with \(A\)-points \(\mathcal{L}^{+}X(A)=X(A[[t]])\)
\item The \emph{loop space}, denoted \(\mathcal{L}X\), is defined as the space with \(A\)-points \(\mathcal{L}X(A)=X(A((t)))\).
\item There is an evident \(\mathbb{G}_{m}\) action on \(\mathcal{L}X\). Functions on \(\mathcal{L}X\) of weight \(w\) for this action will be said to be of \emph{conformal weight} \(w\).
\end{enumerate}\end{definition}

We record some simple and well-known facts below. For a more detailed introduction to these objects we recommend \cite{KV} and its sequels. 

\begin{lemma} \begin{enumerate} \item If \(X\) is a scheme then the functor \(\mathcal{L}^{+}X\) is representable by a scheme.
\item If \(X\) is moreover finite type and affine, then the functor \(\mathcal{L}X\) is representable by an ind-scheme, and in fact by an ind-affine scheme, that is to say a filtered colimit of affine schemes along closed embeddings. 
\item If \(X\) is smooth then both the loop and the arc space are formally smooth. \end{enumerate}
\end{lemma}

\begin{remark} Giving an ind-affine scheme is equivalent to giving a complete and \emph{linearly topologized} \(\mathbb{C}\)-algebra, that is a complete topological algebra with a neighbourhood basis at \(0\) consisting of ideals.\end{remark}
\begin{example}The case of \(\mathcal{L}\mathbb{A}^{1}\) corresponds to the linearly topologized algebra $\lim_{n}\,\mathbb{C}[\mathbb{Z}_{\geq -n}]$. The function \(z_{i}\) is of conformal weight \(i\). The topology is such that for example the infinite sum $\sum_{j}z_{j}z_{-j}$ is convergent. \end{example}
We now introduce the functions on the loop space that we wish to consider. We will assume \(X\) is an affine scheme in what follows. If \(X=spec(A)\), then we will  write $$X\widehat{\times}D^{*}=spec(A((t))).$$ We extend the definition of \((-)\widehat{\times}D^{*}\) to ind-affine schemes by imposing compatibility with filtered colimits. Note that if an ind-affine scheme \(\mathcal{X}\) corresponds to a linearly topologized algebra \(A\), then the ind-affine scheme \(\mathcal{X}\widehat{\times}D^{*}\) corresponds to the algebra, \(A\{t\}\), of \emph{topological Laurent series} over \(A\), i.e. the algebra of formal series in \(t\) and \(t^{-1}\) with negative tails tending to \(0\); \(\{\sum_{i}a_{i}t^{i}\mid a_{i}\rightarrow 0, i\rightarrow -\infty\}\), this is of course the natural completion of the naive algebra of $A$-valued Laurent series. 
\begin{definition} The functor of points  definition of \(\mathcal{L}X\) produces a universal map $$\mathcal{L}X\widehat{\times}D^{*}\longrightarrow X.$$ We call this the \emph{evaluation map} and denote it $\eta$. If \(f\) is a function on \(X\), we denote by $\Lambda(f)$ the constant term of $\eta^{*}(f)$. \end{definition}
\begin{example} In the case of \(\mathbb{A}^{1}\), $\eta$ can be understood very easily. Indeed, by definition it is a global function on $\mathcal{L}\mathbb{A}^{1}\widehat{\times}\,D^{*}$, the algebra of such is the algebra of topological power series over thecomplete  linearly topologized algebra $\lim_{n}\,\mathbb{C}[\mathbb{Z}_{\geq -n}]$. It is then readily seen that $\eta$ corresponds to the (universal) topological Laurent series $$z(t):=\sum_{i}z_{i}t^{i}.$$ As such we see for example that $$\Lambda(z^{2})=\sum_{i}z_{i}z_{-i}.$$ In general, working with coordinates \(\{z^{j}\}_{j}\), one substitutes the sums \(z^{j}(t)=\sum_{i}z^{j}_{i}t^{i}\) into \(f\) and formally expands, then takes the constant term. \end{example}

\subsection{Homology and Cohomology of Ind-Schemes} If $\mathcal{Y}$ is an ind-scheme space, then the set of $\mathbb{C}$ points of $\mathcal{Y}$ is topologized as the colimit of spaces $Y(\mathbb{C})$, with $Y\rightarrow\mathcal{Y}$ a closed embedding of a scheme $Y$ into $\mathcal{Y}$.

\begin{example} Consider the ind-scheme $\mathbb{A}^{\infty}=\lim_{n}\,\mathbb{A}^{n}$, and let $\mathcal{Y}=\mathbb{A}^{\infty}\setminus\{0\}$. Then $\mathcal{Y}(\mathbb{C})$ is the weakly contractible space of non-zero finitely supported sequences of elements of $\mathbb{C}$. \end{example}

We remark that the above example illustrates a fairly typical phenomenon, the support of the reduced homology of the spaces in a colimit presentation of the ind-scheme $\mathcal{Y}$ tends to infinity, and so nothing survives to the limit. Note that the same thing happens for cohomology, as the relevant $\mathbb{R}\lim$ terms vanish in such cases, as the resulting pro-systems are manifestly Mittag-Leffler. One might then hope that a suitable renormalization procedure could remedy this. In \cite{AK} the authors note that such is possible for sufficiently well behaved ind-schemes. 

\begin{definition} If \(X\) is a scheme then we call it \emph{smooth} if it can be represented as the filtered limit of smooth finite type schemes. \end{definition}
\begin{remark} This implies formal smoothness of \(X\). \end{remark} 

We extend this to ind-schemes in an evident manner:
\begin{definition} If \(\mathcal{X}\) is an ind-scheme, we call it \emph{smooth} if it can be represented as the filtered colimit of smooth schemes in the sense of the definition above. \end{definition} 
\begin{remark} This is a strong condition, and indeed is not fulfilled by the loop space to an arbitrary smooth scheme, for example $\mathcal{L}\mathbb{G}_{m}$ is non-reduced, and thus certainly not smooth in this sense. Note that on the other hand $\mathcal{L}\mathbb{A}^{1}$ certainly is smooth. \end{remark}

Henceforth, for a scheme \(X\), when we write \(H^{*}(X)\) it is to be understood as \(H^{*}(X(\mathbb{C}),\mathbb{C})\). As explained in \cite{AK}, a codimension \(d\) embedding, \(X\rightarrow Y\), of smooth and possibly infinite dimensional schemes induces a \emph{Gysin map} on cohomologies, \(H^*(X)\rightarrow H^{*+2d}(Y)\). We require one more definition:
\begin{definition} Let \(\mathcal{Y}=\underrightarrow{\lim}_{j}\mathcal{Y}^{j}\) be an ind-scheme. Then a \emph{dimension theory}, \(\Delta\), for \(\mathcal{Y}\) is a rule assigning an integer \(\Delta(j)\) to each scheme \(\mathcal{Y}^{j}\) in such a way that for \(i<j\), the codimension of \(\mathcal{Y}^{i}\) inside \(\mathcal{Y}^{j}\) is equal to \(\Delta(j)-\Delta(i)\). \end{definition} 
\begin{remark} \begin{itemize} \item It is clear that a  dimension theory always exists for \(\mathcal{Y}\) and the set of such is a \(\mathbb{Z}\)-torsor for connected $\mathcal{Y}$. \item Given two presentations of \(\mathcal{Y}\) as a colimit of schemes along finite codimensional closed embeddings, a dimension theory with respect to one presentation canonically produces a dimension theory with respect to the other. \end{itemize}\end{remark}

We require one final piece of data, here we assume that we are given a smooth ind-scheme $\mathcal{Y}$ equipped with a global function $F:\mathcal{Y}\rightarrow\mathbb{A}^{1}$.

\begin{definition} We say that $F$ is smooth if there is a presentation of $\mathcal{Y}$ as an ind-limit of smooth closed subschemes on which $F$ is smooth in the ususal sense. \end{definition}

\begin{remark} If $F$ is smooth over $\mathbb{G}_{m}$ then the nearby fibres are endowed with a smooth structure as ind-schemes. \end{remark}

We are finally in the position to define the renormalized cohomology of a smooth ind-scheme, \(\mathcal{Y}\), equipped with a dimension theory \(\Delta\).
\begin{definition} We define \(H^{*}_{\Delta}(\mathcal{Y})=\underrightarrow{\lim}_{j}H^{*+2\Delta(j)}(\mathcal{Y}^{j})\), where \(\{\mathcal{Y}^{j}\}_{j}\) is a given smooth presentation of \(\mathcal{Y}\) and the limit is taken with respect to the Gysin maps. In the presence of a global function $F$ which is smooth over $\mathbb{G}_{m}$ in the above sense, the natural renormalized cohomology of the nearby fibre will be referred to as \emph{renormalized nearby cohomology}, and denoted $H^{*}_{\Delta}(\mathcal{Y},F)$.\end{definition}
\begin{remark} This doesn't depend on the presentation.  The dependence on the choice of dimension theory is only up to a shift in degree. We will often drop the \(\Delta\) subscript in what follows, and instead write \(H^{*}_{ren}\), which we'll refer to as \emph{renormalised cohomology}.\end{remark}

\begin{remark} We remark that the definition given above makes perfect sense for $*<0$, if we agree that for a topological space $Y$, $H^{<0}(Y):=0$. Further, the definition above may be non-zero for negative $*$. Indeed let us fix the unique dimension theory on $\mathbb{A}^{\infty}\setminus\{0\}$ such that $\Delta(\mathbb{A}^{1})=0$. Then there is a stable class in $H^{-1}_{\Delta}$ corresponding to the natural classes in $H^{2d-1}(\mathbb{A}^{d}\setminus\{0\})$. Now, in this case the choice of another dimension theory allows us to fix this. Note however that since dimension theories form a $\mathbb{Z}$ torsor, if there are non-zero classes in arbitrarily negative degrees with respect to one dimension theory, then this is true with respect to any other dimension theory. Such can certainly be the case, as the example of $\mathbb{P}^{\infty}$ shows. A hyperplane class in $H^{2}(\mathbb{P}^{n})$ can easily be checked to survive in the colimit, and  as $n$ tends to infinity, these live in more and more negative degrees. \end{remark}

\section{The Computation}
\subsection{The Set-Up} For ease of reference we recall the set-up. Let \(F\in\mathcal{O}(\mathbb{A}^{d})\) be a homogenous polynomial with isolated singularity at \(0\). The degree of homogeneity of \(F\) will be denoted \(\delta\). The coordinates on the affine space \(\mathbb{A}^{d}\) will be denoted with upper indices, \(\{z^{1},...,z^{d}\}\). The loop space \(\mathcal{L}\mathbb{A}^{d}\) is then \(\underrightarrow{\lim}_{n}\,\mathbb{A}\{z^{i}_{j}\mid i\in\big[1,d\big], j\in\mathbb{Z}_{\geq -n}\}\) and it carries the function \(\Lambda(F)=CT(F(z^{1}(t),...,z^{d}(t)))\), where \(CT\) denotes the constant term and \(z(t)=\sum_{i}{z_{i}}t^{i}\). The closed subscheme \(\{\Lambda(F)=1\}\) is denoted \(\mathcal{X}\). 

Let us say something about the smooth structure: we write \(\mathcal{X}^{n}_{m}=\mathcal{X}\cap\mathcal{L}^{n}_{m}\mathbb{A}^{d}\), where \(\mathcal{L}^{n}_{m}\mathbb{A}^{d}\) is defined to be \(\mathbb{A}\{z^{i}_{j}\mid i\in\big[1,d\big], j\in \big[-n,m\big]\}\). Evidently we have that $$\mathcal{X}=\underrightarrow{\lim}_{n}\,\underleftarrow{\lim}_{m}\,\mathcal{X}^{n}_{m}.$$ We now record a trivial lemma, guaranteeing that we may define the renormalised cohomology of \(\mathcal{X}\). 
\begin{lemma} The family \(\{\mathcal{X}^{n}_{m}\}_{n,m}\) gives a smooth presentation for \(\mathcal{X}\), i.e. every space \(\mathcal{X}^{n}_{m}\) is smooth. \end{lemma}
\begin{proof} Recall that \(F\) is homogeneous of weight \(\delta >0\) and hence the \(\mathbb{G}_{m}\) action on \(\mathbb{A}^{d}\) induces one on \(\mathcal{L}\mathbb{A}^{d}\) for which \(\Lambda(F)\) has weight \(\delta\). Each \(\mathcal{X}^{n}_{m}\) is then the nearby fibre of a homogeneous function on an affine space. However Euler's formula guarantees that the singularities of such a function are in the central fibre, so we are done. \end{proof}
We will write \(\mathcal{X}^{n}=\underleftarrow{\lim}_{m}\mathcal{X}^{n}_{m}\) in what follows. We have seen above that \(\mathcal{X}^{n}\) is a smooth (infinite type) scheme and of course the colimit of the \(\mathcal{X}^{n}\) is \(\mathcal{X}\).\subsection{Proof of Main Theorem}\begin{proof} We'll break the proof up into a few steps.

(\emph{Step 1}) Observe that if we bound the conformal degree of our variables below by \(-n\), then homogeneity of \(F\) implies that no variables of conformal degree strictly greater than \(n(\delta -1)\) can show up in \(\Lambda(F)\) (more properly \(\Lambda(F)\) restricted to \(\mathcal{X}^{n}\)). This is simply because the conformal weight would become strictly positive otherwise. This immediately implies that we have an isomorphism: \[\mathcal{X}^{n}\xrightarrow{\sim}\mathcal{X}^{n}_{n(\delta-1)}\times\mathbb{A}\{z^{i}_{j}\mid i\in \big[1,d\big], j>n(\delta-1)\}.\] 

(\emph{Step 2}). We wish to calculate the Gysin maps in cohomology for the inclusions \(\mathcal{X}^{n}\hookrightarrow\mathcal{X}^{n+1}\). Now, assume given two smooth finite type spaces \(X\) and \(Y\) and a closed inclusion (relative to \(\mathbb{A}\{\mathbb{Z}_{\geq 0}\}\)), \(X\times\mathbb{A}\{\mathbb{Z}_{\geq 0}\}\hookrightarrow Y\times\mathbb{A}\{\mathbb{Z}_{\geq 0}\}\). Then to compute the Gysin map in cohomology for this inclusion we need only do it for \(X\hookrightarrow Y\).\\ \\ Using (\emph{step 1}) we see that we need to compute the Gysin map for the closed inclusion \(\mathcal{X}^{n}_{(n+1)(\delta -1)}\hookrightarrow\mathcal{X}^{n+1}_{(n+1)(\delta -1)}\). We denote by \(\mathcal{U}^{n}\) the open complement to this inclusion. So as to avoid having to write \((n+1)(\delta -1)\) repeatedly we hereby denote this number \(N\). Recall the Gysin exact sequence: \[\dotsm\rightarrow H^{*}(\mathcal{X}^{n}_{N})\rightarrow H^{*+2d}(\mathcal{X}^{n+1}_{N})\rightarrow H^{*+2d}(\mathcal{U}^{n})\rightarrow H^{*+1}(\mathcal{X}^{n}_{N})\rightarrow\dotsm;\] we'd like to use this to compute the Gysin maps by computing the cohomology of \(\mathcal{U}^{n}\). 

(\emph{Step 3}) Here we compute the cohomology of \(\mathcal{U}^{n}\) explicitly. This relies crucially on the hypothesis that \(F\) be a homogenous isolated singularity.
We claim that \(\mathcal{U}^{n}\) has the same cohomology as \(\mathbb{A}^{d}\backslash\{0\}\), i.e. the cohomology of \(S^{2d-1}\).  In order to prove this we firstly define \(\mathcal{U}^{n}(j)\) to be the locus of non-vanishing of \(\partial_{z^{j}_{N}}\Lambda(F)\). Monomials showing up in \(\Lambda(F)\) can contain at most one variable of conformal degree \(N\) for obvious degree reasons. This implies that \(\Lambda(F)\) looks like \[\sum_{j}z^{j}_{N}(\partial_{z^{j}_{N}}\Lambda(F))+\mathcal{O}(<N),\] where the big \(\mathcal{O}\) notation denotes a sum of monomials with no conformal degree \(N\) terms. Further, let us note that the chain rule for differentiation  implies that we have $\partial_{z^{j}_{N}}\Lambda(F)$ is the $t^{-N}$ coefficient of $\partial_{j}F(z^{1}(t),...,z^{d}(t))$. Now homogeneity of $F$ is used to note that we thus have $$ \partial_{z^{j}_{N}}\Lambda(F)=\partial_{j}F(z_{-n}^{1},...,z^{d}_{-n}).$$

We claim further that \(\mathcal{U}^{n}=\bigcup_{j}\mathcal{U}^{n}(j)\). To see this, recall that \(\mathcal{U}^{n}(j)\) is the locus of non-vanishing of \(\partial_{j}F(z^{1}_{-(n+1)},...,z^{d}_{-(n+1)})\) and the simultaneous vanishing locus of these partials is precisely the locus of vanishing of all the functions \(z^{i}_{-(n+1)}\), since \(F\) was assumed to have an isolated singularity at the origin. That is to say the locus of vanishing is precisely \(\mathcal{X}^{n}_{N}\), which is the complement to \(\mathcal{U}^{n}\), and so the claim holds. Now consider the map $$\mathcal{U}^{n}\longrightarrow\mathbb{A}^{d}\setminus\{0\},$$ defined as $(z^{1}_{-n},...,z^{d}_{-n})$. We claim that this map is a Zariski locally trivial affine space bundle. Indeed, let us cover $\mathbb{A}^{d}\setminus\{0\}$ by opens $\mathcal{V}(j)$ defined by the non-vanishing of $\partial_{j}F$. Recalling that we have $$\Lambda(F)=\sum_{j}z^{j}_{N}(\partial_{z^{j}_{N}}\Lambda(F))+\mathcal{O}(<N),$$ we easily see that $$\mathcal{U}^{n}(j)\longrightarrow\mathcal{V}(j)$$ is a globally trivial affine space bundle with fibre $\mathbb{A}\{z^{i}_{j}\mid (i,j)\in[1,d]\times [-n,N-1]\}$. This suffices to prove that $\mathcal{U}^{n}$ has the cohomology of a $2d-1$-sphere, independent of $n$.

\emph{Step 4} We can now conclude the proof quite easily. We begin with $*>0$. Examining the Gysin sequence  \[\dotsm\rightarrow H^{*}(\mathcal{X}^{n}_{N})\rightarrow H^{*+2d}(\mathcal{X}^{n+1}_{N})\rightarrow H^{*+2d}(\mathcal{U}^{n})\rightarrow H^{*+1}(\mathcal{X}^{n}_{N})\rightarrow\dotsm\]  we see stabilisation in the limit defining \(H^{*}_{ren}\) at \(n=0\) for any \(*>0\), because \(H^{>2d-1}(\mathcal{U}^{n})=0\) by the result above. By Milnor's results (cf. \cite{Milnor}) this will produce Milnor fibre cohomology concentrated in \(*=d-1\). 

The same argument implies stabilisation at \(n=1\) for \(*=0\). It then suffices to prove that this is \(0\), i.e. that \(H^{2d}(\mathcal{X}^{1})=0\). Looking at the Gysin sequence we must show that the residue map $$H^{2d-1}(\mathcal{U}^{n})\longrightarrow H^{0}(\mathcal{X}^{1})$$ is non-zero, as it is then an isomorphism. A representative for the non-zero class in $H^{2d-1}(\mathcal{U}^{n})$ is given with respect to the Cech cover $$\mathcal{U}^{n}=\bigcup_{j}\mathcal{U}^{n}(j),$$ by $d\log z^{1}_{-n}...d\log z^{d}_{-n}$ restricted to $\cap_{j}\mathcal{U}^{n}(j)$, which manifestly has non-vanishing residue, whence we are done. Finally, the case of $*<0$ splits into two subcases, according to whether $2d$ divides $*$. One takes the minimum $n>0$ so that $*+2nd$ is non-negative and applies arguments exactly as above to conclude stabilization at said $n$ if $*$ is not divisible by $2d$ and at $n+1$ otherwise, the vanishing is proven similarily. \end{proof}

We can now prove the vanishing of nearby cycles cohomology for the function $\widehat{F}$. We note that the central fibre is contractible so that we need only prove that the reduced cohomology of the nearby fibre vanishes. 

\begin{theorem} The reduced (co)homology of the ind-scheme $\mathcal{X}$ vanishes.\end{theorem}

\begin{proof} One sees in the course of the above proof that we have that $\mathcal{X}^{n}$ has no reduced cohomology below degree $2(n+1)d-1$ and so the result is proven. \end{proof}

\end{document}